\documentclass[12pt]{elsarticle}

\usepackage{fullpage}
\usepackage[labelsep=period]{caption}[2008/03/29]
\usepackage{wrapfig}    
\usepackage{graphicx}
\usepackage{amsmath}
\usepackage{amssymb}
\usepackage{amsthm}

\DeclareMathOperator{\const}{const}
\DeclareCaptionLabelSeparator{period}{.}

\begin{document}

\renewcommand{\textfraction}{-.2}
\renewcommand{\topfraction}{1.0}

\makeatletter
\def\ps@pprintTitle{%
     \let\@oddhead\@empty
     \let\@evenhead\@empty
     \def\@oddfoot{\footnotesize\itshape
       Preprint submitted to \ifx\@journal\@empty Elsevier
       \else\@journal\fi\hfill\today}%
     \let\@evenfoot\@oddfoot}
\makeatother
\newtheorem{prop}{Proposition}
\newtheorem*{thm}{Theorem}

\theoremstyle{definition}
\newtheorem{defn}{Definition}

\newsavebox{\TmpBox}
\newcommand{\tmp}{}
\newlength{\tmplength}                    %

%
\newcommand{\Equa}[2]{\begin{equation}#2\label{#1}\end{equation}}
\newcommand{\equa}[1]{\[ #1 \]}

\newcommand{\refeq}[1]{{\rm(\ref{#1})}}
\newcommand{\refeqeq}[2]{{\rm(\ref{#1},\ref{#2})}}
%
%
\newcommand{\myfrac}[2]{{\ifmmode{}^{#1}\!/_{\!#2}\else${}^{#1}\!/_{\!#2}$\fi}}
%
%
\newcommand{\abs}[1]{\left\lvert#1\right\rvert}
\renewcommand{\Re}[1]{\mathop{\rm Re}\nolimits\,(#1)}
\renewcommand{\Im}[1]{\mathop{\rm Im}\nolimits\,(#1)}
\newcommand{\sign}{\mathop{\rm sign}\nolimits}
\newcommand{\Deg}[1]{{\ifmmode{#1}^\circ\else${#1}^\circ$\fi}}
\newcommand{\e}{{\mathrm{e}}}
\newcommand{\Exp}[1]{\e^{#1}}
%
\newcommand{\iu}{{\mathrm{i}}}
%
\newcommand{\diff}[1]{{\mathrm d}#1}
\newcommand{\D}[2]{#1_{#2}^{\prime}}
\newcommand{\DD}[2]{#1_{#2}^{\prime\prime}}
\newcommand{\Int}[4]{\displaystyle\int\limits_{#1}^{#2}{#3}\,{\mathrm d}#4}
\newcommand{\IntF}[5]{\displaystyle\int\limits_{#1}^{#2}\dfrac{#3\,{\mathrm d}#5}{#4}}
\newcommand{\Dfrac}[2]{\dfrac{{\mathrm d}#1}{{\rm d}#2}}
\newcommand{\Pd}[2]{\dfrac{\partial#1}{\partial#2}}
\newcommand{\Arc}[1]{\displaystyle{\buildrel\,\,\frown\over{#1}}}
\newcommand{\kA}{k_{_A}}
\newcommand{\kB}{k_{_B}}

\newcommand{\ieq}{\,{=}\,}
\newcommand{\ineq}{\,{\not=}\,}
\newcommand{\ilt}{\,{<}\,}
\newcommand{\igt}{\,{>}\,}
\newcommand{\In}{\,{\in}\,}
\renewcommand{\le}{\leqslant}
\newcommand{\ile}{\,{\le}\,}
\renewcommand{\ge}{\geqslant}
\newcommand{\ige}{\,{\ge}\,}

\newcommand{\Eqref}[1]{\stackrel{\refeq{#1}}{=}}
\newcommand{\eqdef}{\stackrel{def}{=}}
%
%
%
\newcommand{\GT}[1]{$#1\igt0$}
\newcommand{\GE}[1]{$#1\ige0$}
\newcommand{\LT}[1]{$#1\ilt0$}
\newcommand{\LE}[1]{$#1\ile0$}
\newcommand{\EQ}[1]{$#1\ieq0$}
\newcommand{\NE}[1]{$#1\ineq0$}
%
\newcommand{\Vect}[1]{\overrightarrow{#1}}
\newcommand{\HM}{\hphantom{{-}}}

%
\newcommand{\acum}{\widetilde\alpha}
\newcommand{\bcum}{\widetilde\beta}

%
\newcommand{\Kls}[1]{{\ifmmode{\cal K}^{\star}_{#1}\else${\cal K}^{\star}_{#1}$\fi}}
\newcommand{\Kl}[1]{{\ifmmode{\cal K}_{#1}\else${\cal K}_{#1}$\fi}}
\newcommand{\Kr}[1]{K(#1)}

%
%
\def\FigDir{}

\newcommand{\Figref}[1]{\ref{F#1}}
\newcommand{\Reffig}[2][]{fig.\;\Figref{#2}\textit{#1}} 
\newcommand{\RefFig}[2][]{Fig.\;\Figref{#2}\textit{#1}} 

\newcommand{\Infigw}[2]{
\includegraphics[width=#1]{\FigDir#2.eps}}

\newcommand{\Infig}[3]{\Infigw{#1}{#2}\caption{#3}\label{F#2}}


%
%
\newcommand{\Lfig}[2]{
\begin{figure}[#1]
\settowidth{\tmplength}{Fig.~9.99.}%
\parbox[b]{\tmplength}{\caption{}\label{F#2}~ ~ ~ ~ ~}%
\addtolength{\tmplength}{-\textwidth}%
\includegraphics[width=-\tmplength]{\FigDir#2.eps}
\vspace{3mm}
\end{figure}
}
\newsavebox{\CapBox}


\newcommand{\Hfig}[2]{
\settowidth{\tmplength}{Fig.999.9}%
\savebox{\CapBox}{\parbox{\tmplength}{\caption{}\label{F#2}}}%
\addtolength{\tmplength}{3mm}%
\raisebox{5pt}{\parbox[b]{\tmplength}{\usebox{\CapBox}}}%
\addtolength{\tmplength}{-\textwidth}%
\setlength{\tmplength}{-\tmplength}%
\addtolength{\tmplength}{-#1}%
\includegraphics[width=\tmplength]{\FigDir#2.eps}
}

\newcommand{\Bfig}[3]{
\parbox[b]{#1}{\Infig{#1}{#2}{#3}}
}

\newcommand{\Ffig}[3]{
   \begin{figure}[#1]%
      \Infig{\textwidth}{#2}{#3}%
   \end{figure}%
}

\newcommand{\Wfigr}[4]{
\par
\setlength{\tmplength}{#2}%
\addtolength{\tmplength}{2pt}%
\begin{wrapfigure}[#1]{r}{\tmplength}%
\Infig{#2}{#3}{#4}%
\end{wrapfigure}%
}


\newcommand{\quo}[1]{``#1''}  

\author{A.~Kurnosenko}
\ead{Alexey.Kurnosenko@cern.ch}
\address{Moscow Engineering Physics Institute (CERN--MEPhI collaboration)}

\title{Applying inversion to construct rational spiral curves}


\begin{abstract}
A method is proposed to construct spiral curves by inversion 
of a spiral arc of parabola. The resulting curve is rational of 4-th order.
Proper selection of the parabolic arc and parameters of inversion
allows to match a wide range of boundary conditions, namely,
tangents and curvatures at  the endpoints, 
including those, assuming inflection.
\\~~\\
{\bf Keywords:}\,
spiral, transition curve, rational curve, inversion.\\
~~\\
{\bf 2000 MSC:} 53A04.
\end{abstract}

\maketitle{}

\section{Introduction}

Constructing curves with monotone curvature function (spirals)
is a well-known problem in CAD applications. 
One application of finding spiral arcs with predefined
boundary conditions is the design of transitions curves,
joining two given curves in G$^2$-continuous manner.

Another application was proposed in~\cite{IntProp}. 
For a spiral, represented by a set of interpolation nodes
with boundary tangents, a region, 
covering all possible instances of spiral splines, was constructed. 
The width of the region allows to estimate the determiness of a curve
by the given point set,
and allows to accord curve design with the requirements,
imposed by manufacturing tolerances.

The recent approaches to the problem
include approximate solutions for small arcs, solutions for specific 
boundary conditions, attempts to select
spiral segments within traditional polynomial or rational curves.
An exact solution was given in \cite{Hermite} in terms of
Cornu spiral, extended by two circular arcs.

The author's contribution to the subject is the study of spirals
as a special class of planar curves \cite{InvInvEn,PomiMain,PomiShort}.
The main results of this study are the necessary and sufficient conditions
of existence of a spiral with given curvature elements at the endpoints. 
They expand some statemets of
W.~Vogt and A.~Ostrowski, well known from~\cite{Guggen},
onto non-conex spiral arcs, i.e. those with inflection
or multiple winding. The requirement of curvature continuity
was also removed, which is the case of circular splines 
with piecewise constant curvature.

Based on this study,
a construction in terms of rational curves was developed and
is being presented in this note.
The construction is straightforward and does not require
any  heuristic optimizational or fitting procedures. 
A wide range of boundary conditions is covered.
A possible development of the method to satisfy any boundary conditions, 
compatible with spirality, is discussed in the last section.

\begin{figure}[t]
\centering%
\Bfig{.9\textwidth}{ParabDemo}{~}
\end{figure}

\subsection{Overview of the method}

The proposed method is illustrated by \RefFig{ParabDemo}.
Suppose, we have to construct
a spiral arc from the point $A_1\ieq(-1,0)$ to $B_1\ieq(1,0)$. 
Required  parameters 
are shown as tangent vectors and the circles of curvature (dashed)
at the startpoint~$A_1$, marked by~$\alpha$ and~$\kA$,
and,   at~the endpoint~$B_1$, by  $\beta$ and~$\kB$.

First, we construct a specially calculated parabolic arc, shown by dotted line, 
inscribed into its control polygon $A_1P_1B_1$.
This is a second order Bezi{\`e}r curve
\Equa{OrigPar}{%
    \left(x(t),y(t)\right)=
    A_1(1{-}t)^2+2P_1(1{-}t)t+B_1t^2,\qquad 0\le t \le 1,
}
with the control point $P_1\ieq(p,q)$.
At first glance, this arc has nothing in common with 
the required curve. The only visible feature is that
it is a spiral: the vertex of parabola, although close to the point $A_1$,
is outside the arc $A_1B_1$, $t_{vx}\not\in(0,1)$.

Second, we apply linear  fractional transformation 
(also known as M\"obius transformation)
\Equa{z2w0}{%
        W(z;z_0) = \dfrac{z+z_0}{1+z_0 z},
        \quad z_0= x_0{+}\iu y_0\not=\pm1,  
}
where $z\ieq x(t){+}\iu y(t)$ is a point on the parabola,
$z_0$~is a complex constant. 
The image of the parabolic arc is the sought for curve.
In this example we have calculated 
$p\simeq{ -0.8845}$, $q\simeq{-0.3033}$, $x_0\simeq 1.0296$, 
$y_0\simeq -0.6727$
 \ for prescribed values $\alpha\ieq{-\Deg{180}}$, $\kA\ieq 2.5$,
$\beta\ieq \Deg{120}$ and $\kB\ieq 0.5$.

The solution is based on the following features:

\begin{itemize}
\item we can vary 4 free parameters, $\{p,q,x_0,y_0\}$,
to satisfy 4 required values $\{\alpha,\beta,\kA,\kB\}$;

\item  M\"obius transformation preserves monotonicity
of curvature; if $(p,q)$ defines a spiral parabolic arc
(without the vertex inside), the transformed curve is a spiral;

\item original curve \eqref{OrigPar} being 2nd order polynomial,
the image is 4th order rational.
\end{itemize}

It is well known \cite{Markushevich}
that transformation~\eqref{z2w0} includes
movements, homothety, symmetry and inversion. The latter
provides the necessary flexibility in modifying the
form of a curve. In \RefFig[b]{ParabDemo}
a curve with inflection was requested ($\alpha\ieq\beta\ieq{-\Deg{40}}$,
$3\ieq\kA \igt 0 \igt\kB\ieq {-2}$)
and constructed by similar deformation of the parabola with the control point~$P_2$.

\begin{figure}[t]
\centering%
\Bfig{.9\textwidth}{4curves}{~}
\end{figure}

\section{Definitions and notation}

An arc of a curve is described by the functions of the parameter~$t$:
\equa{%
   x(t),\quad y(t),\quad z(t)\ieq x(t)+\iu y(t),\quad \tau(t),\quad k(t),
   \quad 0\ile t \ile 1,
}
$\tau$ and $k$ being the angle of the tangent vector and the curvature
at the point $(x,y)$.
In this article we deal with curves whose function $k(t)$
is monotone on the segment $t\in[0,1]$, i.e. with spiral arcs. 

\begin{defn}
\label{DefShort1}
A spiral arc $\Arc{AB}$ is {\em short}, 
if its tangent vector never achieves
the direction $\displaystyle{\buildrel\,\,\to\over{BA}}$,
opposite to the direction of its chord,
except, possibly, at the endpoints.
\end{defn}

Another definition, namely, that ``a spiral arc is short, 
if it does not intersect the complement of the chord
to the infinite straight line (possibly intersecting the chord itself)'',
is equivalent to Def.\,\ref{DefShort1} \cite[lemma\,4]{PomiMain}.
Four spiral arcs $A_iM_iB_i$, $i\ieq 1,2,3,4$, in \RefFig{4curves} are short.
In this article we consider short spirals only, which is sufficient
for most of CAD applications.
Examples of long ones are:
\setlength{\unitlength}{1pt}%
\begin{picture}(42,10)(0,3)\put(0,0){\Infigw{42pt}{LongExa1}}\end{picture},
\begin{picture}(42,10)(0,3)\put(0,0){\Infigw{42pt}{LongExa2}}\end{picture},
\begin{picture}(32,10)(0,3)\put(0,0){\Infigw{32pt}{LongExa3}}\end{picture}.

We denote below a circle of curvature as a quadruple of coordinates 
of a fixed point, tangent and signed curvature at this point:
\equa{%
   \Kl{i}=K(x_i,y_i,\tau_i,k_i).
}
It may be a straight line, if \EQ{k_i}.

To consider properties of the arc $\Arc{AB}$ with respect to the chord 
$\Vect{AB}$ of the length $|AB|\ieq 2c$,
we choose the coordinate system such that the chord
becomes the segment $[-c,c]$ of the $x$-axis. 
With 
$\alpha\ieq\tau(0)$, $\kA\ieq k(0)$, and 
$\beta\ieq\tau(1)$, $\kB\ieq k(1)$,
the boundary circles of curvature take form
\Equa{K1K2cc}{%
  \Kl{1} = K(-c,0,\alpha,\kA),\qquad
  \Kl{2} = K(c,0,\beta,\kB).
}
It is convenient to assume homothety
with the coefficient $c^{-1}$, and to operate
on the segment $[-1,1]$ of the $x$-axis. 
The coordinates $x,y$, arc length~$s$, and curvature~$k$ become
normalized dimensionless quantities, corresponding to $x/c$, $y/c$, $s/c$ 
and $kc$.
With such homothety applied, 
boundary circles~\refeq{K1K2cc} appear as~\refeq{K1K2c1}.
\begin{defn}
A spiral arc with boundary curvature elements
\Equa{K1K2c1}{%
  \Kl{1} = K(-1,0,\alpha,a),\qquad
  \Kl{2} = K(1,0,\beta,b),
       \qquad   a = c\,\kA,\quad  b = c\,\kB\,.
}
is said to be in \textit{normalized} position.
 The product~$kc$, invariant under homotheties,
will be referred to as \textit{normalized curvature}.
\end{defn}

For each curve $A_iM_iB_i$ in \RefFig{4curves}
two circular arcs are  traced from~$A_i$ to~$B_i$.
They share tangent with the spiral at one of the endpoints.
These two arcs form a {\em lense}, enclosing the spiral.
The angular width of the lense (signed) is 
\equa{%
    \sigma=\alpha+\beta.
}

Four points $M_i$  in  \RefFig{4curves}, 
four arcs $A_iM_iB_i$ (the first one is parabola) and four lenses are images of
each other under transformation~\eqref{z2w0}. 
The angle~$\sigma$ remains constant for the four lenses.

\section{Theoretical background}

The following is valid for spirals:
\def\labelenumi{(\theenumi)}        
\def\theenumi{\roman{enumi}}
\begin{enumerate}
\item \label{TB:Inversion}
Monotonicity of curvature is preserved under inversion,
increasing curvature being transformed into decreasing, and vice versa 
\cite[theorem\,1]{InvInvEn}.

\item \label{TB:Qlt0}
The necessary and sufficient condition for the existence of a non-biarc spiral
with boundary conditions \eqref{K1K2cc}
is inequality
\Equa{Qlt0}{%
   Q(\Kl{1},\Kl{2}) = 
     (\kA c+\sin\alpha)(\kB c-\sin\beta)+\sin^2\dfrac{\alpha+\beta}{2}<0
}
\hbox{\cite[th.\,2]{PomiMain}}.
If \EQ{Q}, the biarc is the unique spiral, matching these boundary
parameters.

\item \label{TB:Vogt}
For the existence of a short spiral, it is required additionally that 
\equa{
     \sign(\alpha+\beta)=\sign (\kB-\kA)\not= 0,
} 
\cite[th.\,1]{PomiShort}%
\footnote{%
Articles \cite{PomiMain,PomiShort} are not yet translated in English.
The interested reader may find
English version of all these facts and proofs in
arXiv:math/0601440v2 ({\em A.~Kurnosenko. Around Vogt's theorem}\,). 
},
or, in details,
\Equa{VogtShort}{%
  \begin{array}{llc}
  \mbox{if~~}\kA <  \kB:\qquad& \alpha{+}\beta>0,\quad &
               -\pi < \alpha,\beta \le \pi; \\
  \mbox{if~~}\kA >  \kB:\qquad& \alpha{+}\beta<0,\quad & 
               -\pi \le \alpha,\beta < \pi.
  \end{array}
}

\item \label{TB:Invariance}
Both $Q$ and $\sigma\ieq\alpha{+}\beta$ are invariant under M\"obius transformations.

\item \label{TB:Lense}
Short spiral is enclosed in its lense \cite[theorem\,2]{PomiShort}.
\end{enumerate}

Below we provide some comments to these statements.

Alternative proof of \refeq{TB:Inversion}
could be easily derived from the well known facts that
(a)~the curve in the vicinity of some point intersects its circle of curvature 
if the curvature is monotone,
and (b)~the curve is located from one side of this circle, 
if the point corresponds to vertex. These local properties are evidently 
invariant under inversions, and no new vertices can appear on the transformed curve.
\smallskip

In \refeq{TB:Qlt0} and \refeq{Qlt0} $Q=\sin^2\frac{\psi}{2}$, 
where $\psi$ is the intersection angle
of the two circles of curvature. \LT{Q} means that~$\psi$
is pure imaginary.
 \LT{Q} can be interpreted as follows:
if two given circles are inverted into a concentric pair, 
taking into account their orientation,
the resulting circles will be parallel 
(and not {\em anti}\,parallel, as it will be the case if $Q\igt 1$).
\EQ{Q} involves \EQ{\psi}, and corresponds to tangency of two circles
(biarc curve can be constructed). 
$0\ilt Q\ilt 1$ means their real intersection, 
which is not  compatible with spirality.
Neither is $Q\ige 1$.
\smallskip

Statement~\refeq{TB:Vogt} is illustrated by \RefFig{ParabDemo} 
for decreasing curvature (\LT{\sigma\ieq\alpha{+}\beta}).
Note that if we bring the function $\tau(t)$ to the range $[-\pi,\pi]$,
then, to preserve continuity, the value~$-\pi$, not~$+\pi$, 
should be assigned to the angle~$\alpha$
of the spiral $A_1B_1$ in  \RefFig[a]{ParabDemo}. 
This accords to~\eqref{VogtShort}.
Curvatures $k(t)$ in \RefFig{4curves} are increasing, and \GT{\sigma}.
\smallskip

For \refeq{TB:Invariance} we note that oriented angles $\psi$ and $\sigma$ 
only change sign under inversions,
wherefrom the invariance of $Q,\abs{\sigma}$ follows.
Transformation~\eqref{z2w0} is either identity (\EQ{z_0}), 
or includes both inversion and symmetry.
The type of monotonicity of curvature  (increasing/decreasing)
is swapped under inversion and restored under symmetry.
So does $\sign\sigma$.
\smallskip

Statement \refeq{TB:Lense} was proven for convex spirals in~\cite{Theorem5}
and for any short spiral in~\cite{PomiShort}. \RefFig{4curves} demonstrates it,
including non-convex case $A_2M_2B_2$.
\smallskip

\RefFig{HypQeq0} illustrates inequality \LT{Q}~\eqref{Qlt0},
for fixed angles $\alpha,\beta$ taking form
\equa{
   Q(a,b;\alpha,\beta) = 
     (a+\sin\alpha)(b-\sin\beta)+\sin^2\dfrac{\alpha+\beta}{2} < 0.
}
Angles $\alpha\ieq\Deg{45}$,
$\beta\ieq{\Deg{-15}}$ are those of the arc $A_1M_1B_1$, the point 
$K\ieq(a,b)\simeq(-1.98,\, -0.10)$
corresponds to normalized boundary curvatures of this arc.
The equation \EQ{Q(a,b)}
describes a hyperbola in the plane $(a,b)$. 
Two convex regions \LE{Q}, bounded by two branches of the hyperbola, 
cover the values of boundary curvatures, 
allowing spiral arcs to be constructed. The upper left branch,
lying in the halfplane $a\ilt b$,
corresponds to increasing curvatures. As the angles in the first example are such that
\GT{\alpha{+}\beta\ieq\Deg{30}}, 
corresponding short spirals are of increasing curvature~\eqref{VogtShort};
therefore, possible boundary curvatures $(a, b)$ of short spirals 
for this case are covered by the upper left region.

\begin{figure}[t]
\centering%
\Bfig{\textwidth}{HypQeq0}{~}
\end{figure}

The second example in \RefFig{HypQeq0} is drawn for 
$\alpha\ieq{-\Deg{180}}$, $\beta\ieq \Deg{120}$, 
the angles of the spiral $A_1B_1$ in \RefFig[a]{ParabDemo}. 
As $\sigma\ieq{\Deg{-60}}<0$, short curves are of 
decreasing curvature, $a\igt b$. Possible values of $(a,b)$ are
covered by the lower right hyperbolic region.

The following inequalities just reflect the position of two regions
in question with respect to the asimptotae of the hyperbola \cite[cor.~2.1]{PomiMain}:
\Equa{Cor21}{%
       a \lessgtr b\qquad\Longrightarrow\qquad
       a  + \sin\alpha \lessgtr0,\quad
       b  - \sin\beta\,\gtrless 0.
}
Curvatures at the center of the hyperbola 
$(a_0\ieq{-\sin\alpha},\:b_0\ieq\sin\beta)$ are those of
two circular arcs, bounding the lense.

Additional hyperbola, drawn dashed, bounds the region of applicability
of the proposed method. 
This is discussed later (Prop.\,\ref{PropQS}).

\section{Transformation of a spiral arc}

Now we investigate map~\eqref{z2w0}.
This is a particular case of general M\"obius transformation
$z\to \frac{az+b}{cz+d}$, $ad{-}bc\neq0$, 
keeping points $(-1,0)$ and $(1,0)$ intact.

\begin{prop}
\label{PropBasic}%
Let two circles of curvature \eqref{K1K2c1}  of a spiral arc, 
and two circles
\equa{
       \Kls{1} = K(-1,0,\alpha^\star,a^\star),\quad  
       \Kls{2} = K(1,0,\beta^\star,b^\star),
}
of another spiral arc are such that
\Equa{sameSigQ}{%
      \sigma(\Kl{1},\Kl{2})=\sigma(\Kls{1},\Kls{2}),\qquad
       Q(\Kl{1},\Kl{2})\ieq Q(\Kls{1},\Kls{2}).
}       
Then transformation~\refeq{z2w0} with
\Equa{z0r}{%
    z_0=\dfrac{r_0\Exp{\iu\lambda_0}-1}{r_0\Exp{\iu\lambda_0}+1},\qquad\mbox{where}\qquad
    \left\{
    \begin{array}{l}
       \lambda_0=\alpha^\star{-}\alpha = \beta{-}\beta^\star\,,\\
       r_0=\dfrac{a  + \sin\alpha}{a ^\star + \sin\alpha^\star}=
        \dfrac{b^\star - \sin\beta^\star}{b - \sin\beta}>0,
    \end{array}\right.
}
maps the fist pair of circles to the second one.
\end{prop}
\begin{proof}[{\bf Proof.}]
Conditions \eqref{sameSigQ} can be written as
\Equa{SameSQ2}{%
    \begin{array}{rl}
     &\alpha+\beta=\alpha^\star+\beta^\star,\\
    &(a + \sin\alpha)(b - \sin\beta)=
     (a^\star + \sin\alpha^\star)(b^\star - \sin\beta^\star);
    \end{array}
}
that's why two expressions for $\lambda_0$ and $r_0$ in \eqref{z0r} are equivalent.
Because $\sign(\alpha{+}\beta)=\sign(\alpha^\star{+}\beta^\star)$,
two spirals have the same type of monotonicity of curvature.
Numerator and denominator in $r_0$ are non-zero and of equal signs~\refeq{Cor21}.
Thus \GT{r_0} in \eqref{z0r} is justified. 

Treating $z$ in $W(z;z_0)$ as any object, 
subjected to map~\refeq{z2w0}, we denote images of circles $\Kl{1,2}$ 
as $W(\Kl{1,2};z_0)$. We have to prove that $W(\Kl{1,2};z_0)=\Kls{1,2}$.

Let $z(s)$ be arc-length parametrization of the circle~\Kl{1}, i.e.
\equa{%
       z(s)=-1+\dfrac{\iu}{a}\Exp{\iu \alpha}(1-\Exp{\iu a s}):\quad
       \D{z}{s}(s)=\Exp{\iu(\alpha+a s)},\quad
       \DD{z}{ss}(s)=\iu a\Exp{\iu(\alpha+a s)}.
}
We need not pay special attention to the straight line case \EQ{a}, because
\equa{%
        \lim_{a\to 0}z(s)=-1+s\Exp{\iu \alpha}
}
is the straight line in question.
Let $w(s)\ieq W(z(s);z_0)$ be some parametrization of the image of~$\Kl{1}$
(in general case the parameter~$s$ is not arc length for~$w(s)$).
Calculating derivatives at $z(0)\ieq w(0)\ieq {-1}$ yields:
\equa{%
   \begin{array}{l}
     \D{w}{s}(s)=\D{z}{s}(s)\dfrac{1-z_0^2}{[1+z_0z(s)]^2},\\
     \D{w}{s}(0) =\D{z}{s}(0)\dfrac{1-z_0^2}{[1+z_0 z(0)]^2}=
     \Exp{\iu\alpha}\dfrac{1+z_0}{1-z_0}=
     \Exp{\iu\alpha}\cdot r_0\Exp{\iu\lambda_0}=
     r_0\Exp{\iu\alpha^\star}.
   \end{array}  
}
As \GT{r_0}, we get required tangent $\alpha^\star$ for~$w(0)$.
To control curvature, we need the second derivative:
\equa{%
   \begin{array}{rcl}
     \DD{w}{ss}(s)&=&\DD{z}{ss}\dfrac{1-z_0^2}{[1+z_0z(s)]^2}
               -2z_0{\D{z}{s}}^2\dfrac{1-z_0^2}{[1+z_0z(s)]^3},\\
     \DD{w}{ss}(0)
         &=&\iu a\Exp{\iu\alpha}\dfrac{1+z_0}{1-z_0}
             -2z_0\Exp{2\iu\alpha}\dfrac{1+z_0}{(1-z_0)^2}=
         r_0\Exp{\iu\alpha^\star}(\iu a+\Exp{\iu\alpha}-r_0\Exp{\iu\alpha^\star}).
    \end{array}
}
We separate real and imaginary parts, 
\ $\D{w}{s}(0)=u_1+\iu v_1$, \ $\DD{w}{ss}(0)=u_2+\iu v_2$:
\equa{%
   \begin{array}{l}
    u_1  =r_0\cos\alpha^\star,\qquad  v_1  =r_0\sin\alpha^\star,\\
    u_2  = \HM{}r_0[\cos(\alpha{+}\alpha^\star)-a\sin\alpha^\star-r_0\cos(2\alpha^\star)],\\
    v_2  = -r_0[\sin(\alpha{+}\alpha^\star)+a\cos\alpha^\star+r_0\sin(2\alpha^\star)].
  \end{array}
}
Calculating the curvature of the circle~$w(s)$  at \EQ{s} proves that
$W(\Kl{1};z_0)=\Kls{1}$:
\equa{%
   \dfrac{v_2u_1-v_1u_2}{(u_1^2+v_1^2)^{3/2}}=
          \dfrac{a+\sin\alpha}{r_0}-\sin\alpha^\star=a^\star.
}

Similar calculations for the second pair of circles can be omitted.
The required values $\beta^\star$ and~$b^\star$ 
result from~\eqref{sameSigQ} and from the invariance of~$Q,\sigma$ 
(\ref{TB:Invariance}).
Namely, denote the image $W(\Kl{2};z_0)$ as $K(1,0,\bar{\beta},\bar{b})$.
Then
\vspace{-.5\baselineskip}
\equa{%
    \sigma(\,\overbrace{W(\Kl{1};z_0)}^{\Kls{1}}\,,W(\Kl{2};z_0))
     \stackrel{\mathrm{(\ref{TB:Invariance})}}{=}
      \sigma(\Kl{1},\Kl{2})
      \stackrel{\eqref{sameSigQ}}{=}\sigma(\Kls{1},\Kls{2}).
}
Thus obtained equation 
$\sigma\left(\Kls{1},W(\Kl{2};z_0)\right)=\sigma(\Kls{1},\Kls{2})$
is in fact 
$\alpha^\star{+}\bar{\beta}=\alpha^\star{+}\beta^\star$,
and yields $\bar{\beta}\ieq\beta^\star$. The same reasoning for~$Q$
yields $\bar{b}\ieq b^\star$, i.e. $W(\Kl{2};z_0)=\Kls{2}$.
\end{proof}

\section{Spiral parabolic arc}

In this section we explore a spiral parabolic arc from the viewpoint of 
its boundary conditions and invariants.
The equation~\eqref{OrigPar} with the control point $(p,q)$
can be rewritten as
\Equa{XYParab}{%
    \begin{array}{l}
      x(t)=-(1{-}t)^2+2p(1{-}t)t+t^2,\\
      y(t)=2qt(1-t),
    \end{array}\qquad q\not=0,\quad 0\le t \le 1.
}
Calculate
\equa{%
   \begin{array}{l}
      x'(t)=2(1{+}p)(1{-}t)+2(1{-}p)t,\quad  y'(t)=2q(1{-}t)-2qt,\\
      g(t)=\sqrt{{x'}^2(t)+{y'}^2(t)}%
          =2\sqrt{(1{-}t)^2 h_1^2+(1{-}t)t(h_1^2{+}h_2^2{-}4)+t^2h_2^2},
   \end{array}
}
where
\Equa{pH1H2}{%
    h_1=\abs{AP}=\sqrt{(1+p)^2+q^2}\quad\mbox{and}\quad
    h_2=\abs{PB}=\sqrt{(1-p)^2+q^2}.
}
Since $\cos\tau(t)=x'(t)/g(t)$, $\sin\tau(t)=y'(t)/g(t)$, 
the boundary angles are defined by
\Equa{ABparab}{%
     \cos\alpha=\dfrac{1+p}{h_1},\quad 
     \sin\alpha=\dfrac{q}{h_1},\qquad
     \cos\beta=\dfrac{1-p}{h_2},\quad
     \sin\beta=\dfrac{-q}{h_2},
}
and invariant $\sigma=\alpha{+}\beta$ \ by
\Equa{ABsigma}{%
   \cos\sigma=\dfrac{1-p^2+q^2}{h_1h_2},\qquad
   \sin\sigma=\dfrac{-2pq}{h_1h_2}.
}
Curvatures are
\Equa{Ktk1k2}{%
     k(t)=\dfrac{y''x'-x''y'}{g^3}=\dfrac{-8q}{g(t)^3},\qquad  
     a=k(0)=\dfrac{-q}{h_1^3} ,\quad
     b=k(1)=\dfrac{-q}{h_2^3}.
}
Invariant \eqref{Qlt0}, expressed as a function of the control point, looks like
\Equa{Qpq}{
   Q(p,q)=\dfrac{1}{2}+\dfrac{p^2+q^2-1}{2h_1h_2}-\dfrac{q^2(2p^2+2q^2+1)}{h_1^3h_2^3}.
}

\begin{prop}
Parabolic arc \refeq{XYParab} is spiral if and only if the control point $(p,q)$
satisfies inequalities
\Equa{NoVertex}{
   \begin{array}{ll}
                      &F_1(p,q)\cdot F_2(p,q)\le 0,\qquad q\not=0,\\
       \mbox{where}\;& F_1(x,y)=x^2+x+y^2,\quad F_2(x,y)=x^2-x+y^2.
   \end{array}
}  
\end{prop}

\begin{proof}[{\bf Proof.}] 
Differentiating of $k(t)$ \refeq{Ktk1k2} yields
\equa{%
   \begin{array}{l}
     k'(t)=\dfrac{192q[(p^2+q^2)(2t-1)-p]}{g(t)^5}%
          =\dfrac{192q}{g(t)^5} \left[-(1-t)F_1(p,q)+tF_2(p,q)\right],\\
     k'(0)=-\dfrac{6q}{h_1^5}F_1(p,q), \qquad k'(1)=\dfrac{6q}{h_2^5}F_2(p,q).
   \end{array}  
}
To avoid vertex to occur in $t\in(0,1)$, we require
that the derivative did not change sign
within this interval, allowing the vertex at $t\ieq0$ or $t\ieq1$:
\GE{k'(0)k'(1)}. This yields the first inequality in~\refeq{NoVertex};
the second one prevents parabola from degenerating into a straight line.
\end{proof}

\begin{figure}[t]
\centering%
\Bfig{\textwidth}{PQCircHyp}{~}
\end{figure}

\RefFig[a]{PQCircHyp} illustrates~\refeq{NoVertex}: the control point should be taken
within or on the boundary of any of two circles $F_{1,2}(x,y)=0$, but not on the $x$-axis.

\begin{prop}
The locus of control points $(p,q)$, yielding $\sigma=\const$, is
\Equa{Hwconst}{%
   H(x,y)\equiv \sin\sigma(1-x^2+y^2)+2x y\cos\sigma=0,\qquad
    \sign(xy)=-\sign\sigma.
}
\end{prop}

\begin{proof}[{\bf Proof}] The proof follows immediately from \refeq{ABsigma}.
\end{proof} 
The locus~\refeq{Hwconst} is an equilateral hyperbola,
centered at the origin, passing through points $A$ and $B$
(\RefFig[b]{PQCircHyp}). 
The second equality in~\refeq{Hwconst}
keeps only the part of the hyperbola, lying, if \GT{\sigma}, in quadrants II,~IV. 
Finally, subarcs $AA_1$ and $B_1B$
provide control points $(p,q)$, corresponding to spiral parabolic arcs.
For \LT{\sigma} we get the picture,  symmetric about the $x$-axis,
with subarcs $AA_1$ and $B_1B$ in quadrants I,~III (\RefFig[c]{PQCircHyp}).

Both conditions \refeq{Hwconst} in polar coordinates $(\rho,\xi)$ look
like the polar equation of the hyperbola and the intervals for~$\xi$:
\Equa{HypPolar}{%
    \rho(\xi)=\sqrt{\dfrac{\sin \sigma}{\sin (\sigma-2\xi)}},\quad
   \begin{array}{lll}
       \sigma>0:& \dfrac\sigma 2-\dfrac\pi 2<\xi<0, 
                &\dfrac\pi 2+\dfrac\sigma 2<\xi<\pi;\\[8pt]
       \sigma<0:& 0<\xi<\dfrac\pi 2+\dfrac\sigma 2,  
                &-\pi<\xi<\dfrac\sigma 2-\dfrac\pi 2 .
   \end{array}
}

\begin{prop}\label{PropQS} Under conditions 
\begin{subequations}\label{QSmax}
   \begin{align}
     & 0\ilt\abs{\sigma_0}\ilt\dfrac{\pi}{2},\label{QSmaxa}\\
     & Q_0\ile Q_{max} =-\dfrac{w^6(w^2+2)}{(1-w^2)(w^2+1)^3},\qquad
         w=\sqrt[3]{\tan\dfrac{\sigma_0}{2}}.\label{QSmaxb}
   \end{align}
\end{subequations}
there exist two parabolic spiral arcs~\eqref{XYParab}, such that $Q(p,q)\ieq Q_0$
and $\sigma(p,q)\ieq \sigma_0$. 
\end{prop}
\begin{proof}[{\bf Proof}]
The direction of tangent to the hyperbola \EQ{H(x,y)}
at points $A$ and $B$ is $\sigma_0$ (\RefFig[b]{PQCircHyp}). Condition~\eqref{QSmaxa}
assures the existence of the non-vanishing arcs $AA_1$ and $BB_1$
within limiting circles~\eqref{NoVertex}.

Consider the behavior of the invariant $Q(p,q)$ while the control point moves 
along the path $AA_1$ (or $BB_1$).  
$Q(p,q)$ can be considered as the function $Q(\xi)$
of the polar angle. 
To get it, we first express the product~$h_1h_2$  from \eqref{ABsigma}:
\equa{%
     h_1h_2= \dfrac{-2pq}{\sin\sigma_0}=
     \dfrac{-2\rho^2(\xi)\sin\xi\cos\xi}{\sin\sigma_0}=
     \dfrac{\sin 2\xi}{\sin (2\xi{-}\sigma_0)}.
}
Substituting $p^2+q^2\ieq \rho^2(\xi)$, $q^2\ieq \rho^2(\xi)\sin^2\xi$, and the above
product into~\eqref{Qpq}, we obtain
\Equa{Qxi}{%
    \begin{aligned}
      Q(\xi;\sigma_0)=&f_2(\sigma_0)  - \sin^3\sigma_0 f_1(\xi),\\
      \mbox{where}\quad&f_2(\sigma)=\cos^3\sigma-\dfrac{3}{2}\cos\sigma+\dfrac{1}{2},\\
      \mbox{and}\quad&f_1(\xi)=\dfrac{1+4\cos^2\xi-8\cos^4\xi}{8\cos^3\xi\sin\xi}
              =\dfrac{\tan^4\xi +6\tan^2\xi-3}{8\tan\xi}.
    \end{aligned}
}
$Q(\xi;\sigma_0)$ is a monotone function of $\xi$:
\equa{%
   \Dfrac{Q}{\xi}=\dfrac{-3\sin^3\sigma_0}{8\cos^4\xi\sin^2\xi}
}
As $\xi$ decreases from $A$ to $A_1$ and farther to the asymptota of the hyperbola, 
$Q$ increases 
\equa{%
          \mbox{from}\quad
    \lim_{\xi\to{\pi-0}}Q(\xi;\sigma_0)=-\infty
    \qquad\mbox{to}\qquad
    \lim_{\xi\to{\frac\pi2+\frac{\sigma_0}2+0}}Q(\xi;\sigma_0)=1\,,
}
remaining still negative at $A_1$, because parabolic arc is still spiral~\eqref{Qlt0}.
We denote this value as $Q_{max}$. To calculate it, find 
points $A_1,B_1$ by intersecting the pair of circles \EQ{F_1(p,q)F_2(p,q)}
with the hyperbola, i.e. from the equations 
\equa{%
     (p^2+q^2)^2-p^2=0,\quad p=\rho(\xi)\cos\xi,\quad q=\rho(\xi)\sin\xi.
}
They are simplified to $\rho^2(\xi)=\cos^2\xi$ and, finally, to
\equa{%
  \tan\xi\left(\tan^3\xi\tan\dfrac{\sigma_0}2+
                3\tan\xi\tan\dfrac{\sigma_0}2+1{-}\tan^2\dfrac{\sigma_0}2\right)=0.
}
We ignore the root \EQ{\tan\xi} (intersection points $A,B$),
and substitute $\tan\dfrac{\sigma_0}2=w^3$:
\equa{%
    \begin{array}{ll}
       &w^3\tan^3\xi+3 w^3\tan\xi+1 -w^6=0\\
       \Longrightarrow&\left(w\tan\xi  +1{-}w^2\right)%
         \left[w^2\tan^2\xi+w(w^2-1)\tan\xi+ w^4{+}w^2{+}1 \right]=0.
    \end{array}
}
The unique real root $\tan\xi_1=\frac{w^2-1}{w}$ gives both points $A_1$ and $B_1$.
Calculating~$Q(\xi_1;\sigma_0)$ yields the expression~\refeq{QSmaxb} for~$Q_{max}$.
\end{proof} 

\Wfigr{12}{.24\textwidth}{Qmax}{~}
The plot $Q_{max}(\sigma)$  is shown in \RefFig{Qmax}.
The region of curvatures, rejected by condition~\eqref{QSmaxb},
was shown in \RefFig{HypQeq0} as the band-like zone between the line \EQ{Q(a,b)}
and the dashed line $Q(a,b)\ieq Q_{max}$.
The narrower is the lense, the narrower is the rejection region
$Q_{max}\ilt Q \ile 0$.
In other words, our method is not applicable,
when boundary circles of curvature  are close to tangency.
\medskip

Now we find the explicit solution~$\xi\ieq\xi_0$ for the equation 
$Q(\xi;\sigma_0)=Q_0$,
i.e. find the control points on hyperbola~\eqref{HypPolar} with $\sigma\ieq\sigma_0$,
yielding the parabolic arc with predefined value
of invariants $\sigma_0,\,Q_0$.

We denote for brevity
\Equa{DefQ1}{%
     Q_1=\dfrac{Q_0-f_2(\sigma_0)}{\sin^3\sigma_0}\,,\qquad
     m=\sqrt[3\:]{1+Q_1^2}\,,\qquad 
     n=\sqrt{m^2{+}m{+}1}\,,
}
and
\equa{%
    \begin{array}{l}
       r_1=\sqrt{m-1}=\dfrac{\abs{Q_1}}{n}\,,\\
       r_2=\sqrt{2n-(m{+}2)}=\dfrac{m\sqrt3}{\sqrt{2n + m{+}2}}\,,\\
       r_{12}=r_2^2-r_1^2=2n-(2m{+}1)=\dfrac{3}{2n+2m+1}\,.
    \end{array}
}
Each of the above definitions for $r$'s is represented
in the alternative form, intended to avoid loss of precision while subtracting 
close positive numbers in calculations. For the same reason, the expression
for $\theta_0$ is splitted in~\eqref{TxiSol},
and $r_2{-}r_1$ is replaced by $\frac{r_{12}}{r_1+r_2}$.

\begin{prop} 
Solutions $\xi_0, \bar{\xi}_0$ of the equation $Q(\xi;\sigma_0)\ieq Q_0$, 
and corresponding control points $(p_{1,2},q_{1,2})$ are given by
\vspace{-.5\baselineskip}
\Equa{TxiSol}{%
  \begin{array}{l}
    \:\theta_0=-r_1\sign Q_1-r_2\sign\sigma_0=\left\{%
       \begin{array}{ll}
          -\,\dfrac{r_{12}}{r_1+r_2}\sign \sigma_0,
            &\mbox{if}\quad\sigma_0 Q_1<0;\\[8pt]
          -(r_1+r_2)\sign \sigma_0,&\mbox{if}\quad\sigma_0 Q_1\ge 0;
       \end{array}\right.\\[12pt]     
    \begin{array}{ll}
      \xi_{0}=\arctan \theta_0,  &
         (p_1,q_1)=\left(\rho(\xi_0)\cos\xi_0,\rho(\xi_0)\sin\xi_0\right),\\[4pt]
      \bar{\xi}_{0}=\xi_{0}{+}\pi, &
         (p_2,q_2)=(-p_1,-q_1).
     \end{array}
  \end{array}
}
\end{prop}
\begin{proof}[{\bf Proof}]
Equation~\refeq{Qxi}, namely $Q_0=f_2(\sigma_0)-\sin^3\sigma_0 f_1(\xi)$,
can be transformed to
\Equa{TxiEq}{%
     \theta^4+6\theta^2+8Q_1\theta-3=0,\quad\mbox{with}\quad \theta\ieq\tan\xi,
}
and $Q_1$, defined by \eqref{DefQ1}.
Following Descartes--Euler's method, 
find the cubic resolvent of \eqref{TxiEq} and its roots:
\equa{%
    (z+4)^3-64(1+Q_1^2)=0, \quad 
    \begin{array}{l}
      z_{1\hphantom{,3}}=-4+4m,\\
      z_{2,3}=-4+4m\Exp{\pm\frac{2}{3}\iu\pi}\,.
    \end{array}
}
Now one should choose signs for $\zeta_i=\pm\sqrt{z_i}$ to satisfy
$\zeta_1\zeta_2\zeta_3=-8Q_1$.
We get it, assuming
\equa{%
   \begin{array}{l}
      \zeta_{1\hphantom{,3}}=-\sign Q_1 \sqrt{z_1}=-2r_1\sign Q_1,\\
      \zeta_{2,3}=+\sqrt{z_{2,3}}%
          =r_2\pm {\iu}\sqrt{2n+m{+}2},
   \end{array}
}
thus obtaining 
$\zeta_1\zeta_2\zeta_3=-8\sign Q_1 \sqrt{m^3-1}=-8\sign Q_1\cdot \abs{Q_1}=-8Q_1$.
So, both real roots of \eqref{TxiEq} are equal to
\equa{
      \theta_{1,2}=\dfrac{1}{2}\zeta_1\pm\dfrac{1}{2}(\zeta_2{+}\zeta_3)
       =-r_1\sign Q_1\pm r_2.
}
and are of opposite signs:
\equa{
   \theta_1 \theta_2=r_1^2-r_2^2=-r_{12}<0.
}
From \eqref{ABsigma} we note that 
\equa{%
  \sin\sigma_0=-\dfrac{2\rho^2(\xi)\sin\xi\cos\xi}{h_1h_2}
  \qquad\Longrightarrow\qquad 
  \sign\sigma_0=-\sign\tan\xi=-\sign \theta.
}
So, the only admissible root~$\theta_0$ from $\theta_{1,2}$ is of the sign, opposite 
to that of~$\sigma_0$:
\equa{%
   \begin{array}{rl}
     \mathrm{if}\quad\sigma_0>0,\; Q_1 < 0: \:& \theta_0=\HM{}r_1-r_2<0;\\
     \mathrm{if}\quad\sigma_0>0,\; Q_1\ge0:   & \theta_0=-r_1-r_2<0;\\
     \mathrm{if}\quad\sigma_0<0,\; Q_1\le0:   & \theta_0=\HM{}r_1+r_2>0;\\
     \mathrm{if}\quad\sigma_0<0,\; Q_1 > 0:   & \theta_0=-r_1+r_2>0.
   \end{array}
}
This selection is unified in~\eqref{TxiSol}.
\end{proof}

Thus found $\xi=\arctan\theta_0$ provides two symmetric parabolas, 
shown in \RefFig[c]{PQCircHyp}.

\section{Algorithm}

Now we summarize step by step the construction of a short spiral arc
with predefined boundary conditions.
\smallskip

\begin{figure}[th]
\centering%
\Bfig{.84\textwidth}{Algorithm}{~}%
\end{figure}

\noindent{\bf Step 1.\,} Transformation to the chord's coordinate system.

Given boundary conditions $\Kl{M}=K(x_1,y_1,\tau_1,k_1)$ at the startpoint~$M$
(\RefFig[a]{Algorithm}),
and $\Kl{N}=K(x_2,y_2,\tau_2,k_2)$ at the endpoint~$N$,
transform them by orthogonal transformation to
\equa{%
  \Kls{1} = K(-1,0,\alpha^\star,a^\star),\qquad
  \Kls{2} = K( 1,0,\beta^\star,b^\star),
}
where
\equa{%
\begin{array}{ll}
   \alpha^\star=\tau_1-\mu,\;&a^\star=k_1c\\
   \beta^\star=\tau_2-\mu,   &b^\star=k_2c.
\end{array}\qquad
\begin{array}{l}
    c=\dfrac12\sqrt{(x_2{-}x_1)^2+(y_2{-}y_1)^2},\\
    \mu = \arg\left[(x_2{-}x_1)+\iu(y_2{-}y_1)\right],
\end{array}
}
The transformed configuration is shown in \RefFig[b]{Algorithm}.
\medskip

\noindent{\bf Step 2.\,} 
Check solvability of the problem. 
\begin{itemize}
\item Calculate $Q_0=Q(\Kls{1},\Kls{2})$ and check inequality~\refeq{Qlt0}.
If \GT{Q_0}, boundary conditions are invalid, irrespective of the proposed method:
no such spiral exists. If \EQ{Q_0}, biarc is unique spiral.
Continue, if \LT{Q_0}.

\item As specified in~\eqref{VogtShort}, bring the boundary angles 
to the halfintervals $(-\pi,\pi]$ or $[-\pi,\pi)$. Namely, 
if $\abs{\alpha^\star}\ieq\pi$ or $\abs{\beta^\star}\ieq\pi$, 
replace $+\pi$ by $-\pi$ for the case of decreasing curvature. 
Calculate $\sigma_0=\alpha^\star{+}\beta^\star$, check condition~\refeq{VogtShort}. 
If it fails, no short spiral exists.

\item Verify conditions \eqref{QSmax}.
If $Q_0>Q_{max}$ or $\abs{\sigma_0}\ge\pi/2$, the proposed algorithm is not applicable.
\end{itemize}

\noindent{\bf Step 3.\,} 
Find spiral parabolic arc~\refeq{XYParab} such that
its boundary tangents $\alpha, \beta$ and curvatures $a, b$
satisfy equations~\refeq{SameSQ2}.
Two solutions, $(p_{1,2},q_{1,2})$, 
are explicitly given above~\eqref{TxiSol}.
The explicit expression for~$Q_1$ in~\eqref{DefQ1} is
\equa{%
     Q_1=\cot\sigma_0+
         \dfrac{(a^\star + \sin\alpha^\star)(b^\star - \sin\beta^\star)}%
              {\sin^3\sigma_0}.
}
Two parabolic arcs are shown in \RefFig[c,d]{Algorithm} by dotted lines.
\medskip

\noindent{\bf Step 4.\,}
Apply this step to each of two 
just constructed parabolic arcs. 

Calculate boundary tangents $\alpha,\, \beta$ from~\refeq{ABparab}
and curvatures $a,\, b$  from~\refeq{Ktk1k2}.
Define $r_0$, $\lambda_0$ from~\eqref{z0r}.
Transforming the parabolic arcs by~\refeq{z2w0}, get the sought for
rational curve:
\equa{%
  \begin{split}
    X(t)+\iu Y(t)
       &=\dfrac{x_0(x^2{+}y^2{+}1)+x(1{+}x_0^2{+}y_0^2)+
         \iu[y(1{-}x_0^2-y_0^2)+y_0(1{-}x^2{-}y^2)] }%
         {1+2x x_0-2y y_0+(x^2+y^2)(x_0^2+y_0^2)}=\\
       &=\dfrac{r_0^2 l_1^2-l_2^2 + %
            2\iu r_0[2y\cos\lambda_0-(x^2+y^2-1)\sin\lambda_0]}%
           {r_0^2 l_1^2
            -2r_0[2y\sin\lambda_0+(x^2+y^2-2)\cos\lambda_0]
            +l_2^2},
  \end{split}  
}
where $x(t),y(t)$ are abbreviated to $x,y$, and
$l_1^2=(x{+}1)^2+y^2$,  $l_2^2=(x{-}1)^2+y^2$.
The second expression is tolerant to the case $z_0\ieq\infty$,
which is described by finite values $r_0,\,\lambda_0$.

Two solutions are shown combined in \RefFig[e]{Algorithm}.
\medskip

\noindent{\bf Step 5, final.\,}
Return to the original coordinate system (\RefFig[f]{Algorithm}).

\begin{figure}[t]
\centering%
\Bfig{.8\textwidth}{AlgKs}{~}%
\\[3\baselineskip]%
\hspace{.1\textwidth}\Bfig{.5\textwidth}{ParabLong}{~}%
\hfill
\Bfig{.175\textwidth}{ExaLemV}{~}\hspace{.1\textwidth}
\end{figure}

\section{Illustrations}

In \RefFig{Algorithm} transition curves, joining two concentric circles,
have been constructed. \RefFig{AlgKs} adds the curvature plots 
with respect to arc length for both solutions.

In \RefFig{ParabLong} the boundary circles are also concentric,
but the tangents at points $A$ and $B$ 
are such that no short joining spiral exist:
condition~\refeq{VogtShort} is violated. 
To construct a long one, we introduce the point~$M$
at the polar halfway from $A$ to $B$, and new intermediate circle of curvature,
concentric to both given ones.
Two shorts spirals are constructed on two chords, $AM$ and $MB$.
Combining them together, we obtain four variants
of the long spiral $AMB$, 
and the curvature remains continuous at~$M$. In particular,
taking radius $R_{_M}=\sqrt{R_{_A}R_{_B}}$, we get similarity of boundary conditions.
Arcs~$MB$ can be obtained just by rotation and homothety, applied to both arcs~$AM$.
\smallskip

\begin{figure}[t]
\centering%
\Bfig{.9\textwidth}{Cornu7}{~}%
\end{figure}

In \RefFig{ExaLemV} symmetric boundary conditions are given:
tangents are parallel, 
and $\kA< 0<\kB=\abs{\kA}$.
Curvature changes sign, inflection point on the transition
curve is therefore required. It does appear in both solutions.

One could expect a symmetric solution, but, instead, we get
the symmetry among two solutions.
The problem could be resolved, as in the previous example,
by introducing the intermediate inflection point at the coordinate origin.
\smallskip


In \RefFig{Cornu7} the long arc $C_0C_7$ of Cornu spiral
is drawn by dots and subdivided into 7 short subarcs, $C_iC_{i+1}$. 
The boundary conditions $(x_i,y_i,\tau_i,k_i)$ 
are borrowed from those of Cornu spiral at points $C_i$.
In this example every point $C_i$ is taken almost as far from $C_{i-1}$ as
limitation  \eqref{QSmaxb} allows.
As before, we get two solutions on every segment $C_iC_{i+1}$.
Differences between two solutions and Cornu spiral itself
are well visible on the segment $C_0C_1$.
This can be avoided, as shown on the right side, by inserting
an intermediate point $C_{1/2}$ between $C_0$ and~$C_1$.

The curvature plots with respect to arc length
for approximations of Cornu spiral are also shown in \RefFig{Cornu7}
for both solutions on every segment of the curve $C_0C_7$.
Abscissas of vertical lines correspond to arc length of Cornu spiral in points~$C_i$.
Small vertical marks separate arcs of approximating curves. 
Proportionality $k(s)\sim s$ of Cornu spiral is well reproduced in approximations.

\begin{figure}[t]
\centering%
\Bfig{.9\textwidth}{InvHyp2}{~}%
\end{figure}

\section{Conclusions}

Note that we get much more flexibilty in form control
by involving spiral arcs of other conics or other spirals:
basic Proposition~\ref{PropBasic} is independent
of the kind of spiral involved.
Arcs within a quarter of an ellipse are spirals,
and, subjected to transformation~\eqref{z2w0}, provide another variety of spiral arcs. 
Being 2-nd order rationals,
they also yield 4-th order rational spirals. So do hyperbolic arcs.

\RefFig{InvHyp2} shows inversions of a half-hyperbola,
traced from the vertex at point $A$ through infinity to the second vertex at~$B$.
Several inversions of this branch of hyperbola are shown,
every example with the circle of inversion.
Inversion with respect to the unit circle produces a curve $AOB$,
known as hyperbolic lemniscate.
Although the original curve
has discontinuity at infinity, this feature disappears under
inversion. The infinite point goes to the center
of inversion and becomes an ordinary,  infinitely differentiable,
point of the curve-image.

Varying eccentricity of conic gives additional possibilities
in constructing rational spirals. 
A {\em family} of curves is expected as a solution.
Starting with hyperbola, one could get a desired symmetric solution 
for the symmetric problem in \RefFig{ExaLemV}.

Preliminary investigation shows
that the region of boundary conditions~\refeq{QSmax}
can be extended to $\abs{\sigma_0}\ile\pi$ and \LT{Q_0}.
Boundary angles of the lemniscale $AOB$ are $\alpha\ieq\beta\ieq\pi/2$, 
its lense is the unit circle, 
and the width of the lense is $\sigma\ieq\pi$.
Taking eccentricity $e\gg 1$, two branches can be made anyhow close to parallel lines
(\RefFig[b]{InvHyp2}).
The image can be anyhow close to a biarc curve, and $Q$ close to~0.

In this article we have considered parabolic arcs only,
thus illustrating the simplest, but sufficiently powerful,
version of the proposed method.


\newcommand{\POMI}{Zapiski nauch. sem. \hbox{POMI}}
\vfill  

\end{document}